\documentclass[11pt]{article}
\usepackage{amsmath, amssymb, amscd, amsthm, amsfonts, subfig}
\usepackage{graphicx}
\usepackage{verbatim}
\usepackage{float}
\usepackage{color,colortbl}
\usepackage{tikz}
\usepackage{pgfplots}
\usepackage{hyperref}

\usepackage[font=small,labelfont=bf]{caption}
\usepackage{titlesec}
\expandafter\def\expandafter\normalsize\expandafter{%
    \normalsize%
    \setlength\abovedisplayskip{5pt}%
    \setlength\belowdisplayskip{5pt}%
    \setlength\abovedisplayshortskip{-3pt}%
    \setlength\belowdisplayshortskip{3pt}%
}
\usepackage[scr=boondox, 
            cal=esstix] 
           {mathalpha}
\newtheorem{theorem}{Theorem}[section] 

\newtheorem{corollary}[theorem]{Corollary} 
\newtheorem{lemma}[theorem]{Lemma} 
\newtheorem{proposition}[theorem]{Proposition}

\newtheorem{conjecture}[theorem]{Conjecture}

\theoremstyle{definition}
\newtheorem{example}{Example}

\theoremstyle{remark}
\newtheorem*{remark}{Remark} 
\newcommand{\floor}[1]{\left\lfloor #1 \right\rfloor}
\newcommand{\ceil}[1]{\left\lceil #1 \right\rceil}

\newcommand\nnfootnote[1]{%
  \begin{NoHyper}
  \renewcommand\thefootnote{}\footnote{#1}%
  \addtocounter{footnote}{-1}%
  \end{NoHyper}
}

\DeclareMathOperator{\outdeg}{outdegree}
\newcommand\preceqdot{\mathrel{\ooalign{$\preceq$\cr
  \hidewidth\raise0.225ex\hbox{$\cdot\mkern0.5mu$}\cr}}}

\title{Chip-Firing on the Lattice of Nonnegative Integer Points}
\author{Ryota Inagaki \and Tanya Khovanova \and Austin Luo}
\date{}

\begin{document}

\maketitle

\begin{abstract}
Chip-firing on a directed graph is a game in which chips, a discrete commodity, are placed on the vertices of the graph and are transferred between vertices. In this paper, we study a chip-firing game on the lattice of nonnegative integer points on the $x$-$y$ plane, where we start with $2^n$ chips at the origin. When we fire a vertex $v$, we send one chip to each out-neighbor. We fire until we reach a stable configuration, a distribution of chips where no vertex can fire.

We study the intermediate firing configuration: a table that assigns to each vertex the total number of chips that reach it. We prove that the nonzero entries of the stable configuration correspond to the odd entries of the intermediate firing configuration. The intermediate firing configuration consists of three parts: the top triangle, the midsection, and the bottom triangle. We describe properties of each part. We study properties of each row and the number of rows of the intermediate firing configuration. We also explore properties of the difference tables, which are tables of first differences of each row of the intermediate firing configuration.
\end{abstract}

\nnfootnote{\emph{2020 Mathematics Subject Classification}:  05A10, 11Y55, 05C57}
\nnfootnote{\emph{Key words and phrases:} chip-firing, Pascal's triangle, binomial coefficients}

\section{Introduction}
Chip-firing is a game on a graph in which a discrete commodity, called chips, is placed on the vertices of the graph. A vertex $v$ fires when it has at least as many chips as neighbors and disperses a chip to each neighbor. In a directed graph, vertices can fire chips only along their out-edges. Chip-firing is closely related to the abelian sandpile, which Bak, Tang, and Wiesenfeld~\cite{PhysRevLett} and Dhar \cite{dhar1999abelian} explored; in this setting, a stack of sand disperses when it exceeds a certain height. The study of chip-firing on undirected graphs originates from works including Anderson, Lovász, Shor, Spencer, Tardos, and Winograd \cite{zbMATH04135751} and Bj\"orner, Lovász, and Shor \cite{MR1120415}. The study of chip-firing on directed graphs was initiated by works such as  \cite{MR1203679} by Bj\"orner and Lovász. Numerous variants of chip-firing have been studied, including Biggs’s dollar game\footnote{One can play a recreational variant of this game at the website \cite{DollarGame}.} \cite{MR1676732} on graphs, Guzmán and Klivans’s chip-firing game on invertible matrices \cite{MR3311336, MR3504984}, Hopkins, McConville, and Propp’s labeled chip-firing game on 1-dimensional lattices \cite{MR3691530}, Musiker and Nguyen's labeled chip-firing game on undirected binary trees \cite{MR4827886, Inagaki2025}, and the authors' labeled chip-firing game on directed trees \cite{MR4887467, inagaki2025permutationbasedstrategieslabeledchipfiring, inagaki2025labeledchipfiringdirectedkary}.

In this paper, we study stable configurations resulting from the chip-firing game on the Hasse diagram of $\mathbb{Z}_{\geq 0}^2$, the lattice of nonnegative integer points on the $x$-$y$ plane. This Hasse diagram is a directed graph whose vertex set is $\mathbb{Z}_{\geq 0}^2$ and edge set is $\{(v_x, v_y) \to (v_x+1, v_y), (v_x, v_y) \to (v_x, v_y+1): v = (v_x, v_y) \in \mathbb{Z}_{\geq0}^2 \}$. It is the Hasse diagram of one of the simplest and most familiar non-tree posets that come to mind. The lattice graph of nonnegative integer points is shown in Figure~\ref{fig:quadrantgraph}. 

\begin{figure}[ht!]
    \centering
    \includegraphics[width=0.35\linewidth]{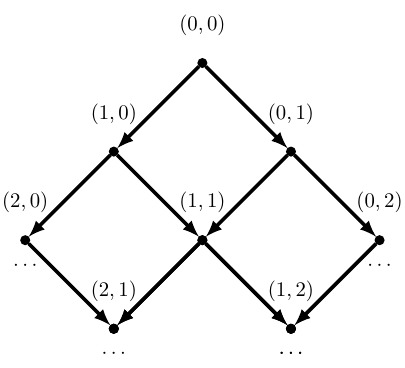}
    \caption{The lattice of nonnegative integer points on the plane with labeled vertices.}
    \label{fig:quadrantgraph}
\end{figure}

In our setup, we start with $2^n$ chips at the origin $(0, 0)$. When a vertex $v=(v_x, v_y)$ has at least two chips, we fire, i.e., take two chips and send one chip to $(v_x+1, v_y)$ and one to $(v_x, v_y+1)$. We play the chip-firing game until no vertex can fire. Once no vertex can fire, we reach a \textit{stable configuration} (see Section~\ref{sec:prelim} for the full definition).

Now we describe the layout of our paper. 

In Section~\ref{sec:prelim}, we introduce preliminary definitions and facts that will be useful in this paper.

Then, in Section~\ref{sec:TotalFirings}, we express the total number of firings required to achieve the stable configuration through the distribution of chips in that configuration.

In Section~\ref{sec:IntermedFiringConfig}, we introduce the intermediate firing configuration $F_n$, a table that describes the total number of chips that reach each vertex in the chip-firing game starting with $2^n$ chips at the root. We then prove that the nonzero entries of the stable configuration correspond to the odd entries in the intermediate firing configuration. We present an example where we start with $2^9$ chips and observe interesting structures in the resulting stable configuration.

In Section~\ref{sec:PascalTriangle}, we observe that the first $n+1$ rows of the intermediate firing configuration $F_n$ are scaled versions of rows of Pascal's triangle.

In Section~\ref{sec:Med}, we prove that in each row of the intermediate firing configuration $F_n$, the nonzero entries first strictly increase towards the middle and then strictly decrease away from it when read from left to right, except that the two central entries are equal when the number of nonzero entries is even. Then we discuss the number of nonzero entries in each row of the intermediate firing configuration, which changes by $+1$ or $-1$ when moving to the next nonzero row. Afterward, we give the upper bound on the number of nonzero rows.

In Section~\ref{sec:BottomTriangle}, we discuss the last rows of the intermediate firing configuration, whose nonzero entries form a triangle. These last rows are minimal in some sense and are explicitly described.

Afterward, in Section~\ref{sec:Discussion}, we summarize our discoveries and show some plots. We also motivate studying the difference tables in the last section. We have extensive computational data that motivate our research. These data can be found in \cite{GithubData}.

Finally, in Section~\ref{sec:DiffTables}, we introduce the difference table $F_n'$, which is the table of differences of consecutive entries in each row of the intermediate firing configuration $F_n$. We prove that the maxima of the nonzero rows of $F'_n$ are nonincreasing and that the entries in the left half of each row first weakly increase and then weakly decrease.

\section{Preliminaries}\label{sec:prelim}

\subsection{Unlabeled Chip-Firing}

In the unlabeled chip-firing game on directed graphs, indistinguishable chips are put on the vertices of a directed graph $G = (V, E)$. Vertex $v$ can fire if and only if $\outdeg(v)>0$ and it has at least $\outdeg(v)$ chips. When it fires, vertex $v$ sends one chip to each out-neighbor and consequently loses $\outdeg(v)$ chips.

A \textit{configuration} $\mathcal{C}$ is a distribution of chips over the vertices of a graph. When the firing process reaches a state in which no vertex can fire, the distribution of chips over the vertices is a \textit{stable configuration}.

A key property of unlabeled chip-firing on directed graphs is the following global confluence property, which is analogous to that of chip-firing on undirected graphs (which is similar to Theorem 2.2.2 of Klivans's textbook \cite{klivans2018mathematics}).

\begin{theorem}[Theorem 1.1 of \cite{MR1203679}]\label{thm:confluence}
For a finite directed graph $G$ and initial configuration $\mathcal{C}$ of chips on the graph, the unlabeled chip-firing game will either run forever or end after the same number of moves and at the same stable configuration. Furthermore, in the terminating case, the number of times each vertex fires is the same regardless of the sequence of firings taken in the game.
\end{theorem}

\subsection{Our Setup}

The \textit{Hasse diagram of the lattice of nonnegative integer points on the plane} is the directed graph whose vertex set is the set of integer points in the first quadrant, i.e., $\mathbb{Z}_{\geq 0}^2 = \{(x, y): x, y \in \mathbb{Z}_{\geq 0}\}$, and whose edges are $\{(x, y) \to (x+1, y): x, y \in \mathbb{Z}_{\geq 0}\}$ and $\{(x, y) \to (x, y+1): x, y \in \mathbb{Z}_{\geq 0}\}$. We say that a vertex $v = (v_x, v_y)$ is in row $i$ if $v_x+v_y=i$. For vertices  $(a, b)$ and $(c, d)$ that are in the same row, we say that $(a, b)$ is \textit{left} of $(c, d)$ if $b-a < d-c$, or equivalently $b < d$. We say that row $i$ is \textit{higher} than row $j$ if $i < j$.

\begin{example}
    Figure~\ref{fig:quadrantgraph} illustrates the Hasse diagram of the lattice of nonnegative integer points. The point $(2, 0)$ is to the left of point $(1, 1)$, since $0 - 2 < 1- 1$ and since $(2, 0)$ and $(1, 1)$ are both in row~$2$. Row $2$ is below row $1$.
\end{example}

Our chip-firing game proceeds on the directed graph described above. We call the vertex $(0,0)$ \textit{the root}.

In our chip-firing game, we denote the \textit{initial configuration} of chip-firing as the placement of $2^n$ chips on the root. A vertex $v=(x, y)$ can \textit{fire} if and only if it has at least $2$ chips. When vertex $v$ \textit{fires}, it transfers a chip from itself to each of its $2$ children, $(x, y+1)$ and $(x+1,y)$.

\begin{example} Figure~\ref{fig:exampleunlabel} shows the unlabeled chip-firing process with $4$ indistinguishable chips initially placed at the origin of the lattice of nonnegative integer points on the plane, where the graph is rotated so edges are directed down.

\begin{figure}[ht!]
\centering
    \subfloat[\centering Initial configuration with $4$ chips]{{\includegraphics[width=0.25\linewidth]{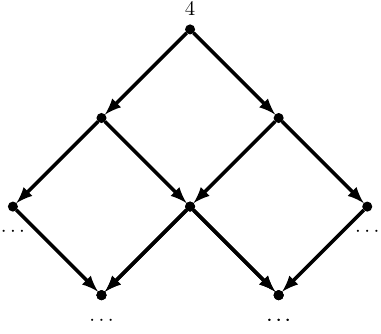} }}%
    \qquad
    \subfloat[\centering Configuration after firing once]{{\includegraphics[width=0.25\linewidth]{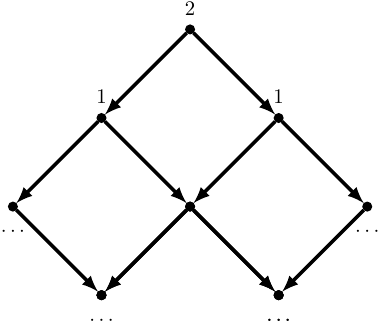} }}%
    \qquad
    \subfloat[\centering Configuration after firing twice]{{\includegraphics[width=0.25\linewidth]{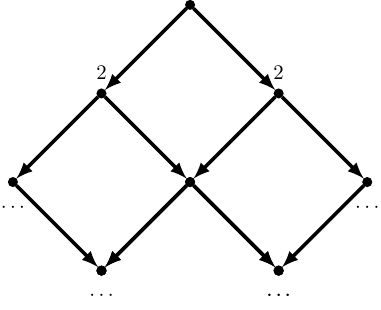} }}%
    \qquad
    \subfloat[\centering Configuration after firing four times]{{\includegraphics[width=0.25\linewidth]{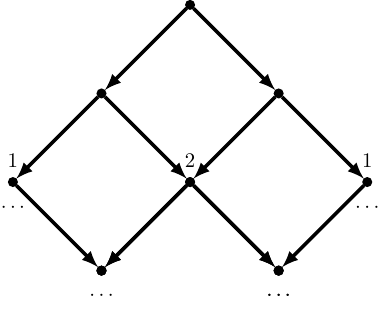} }}
    \qquad
    \subfloat[\centering Stable configuration]{{\includegraphics[width=0.25\linewidth]{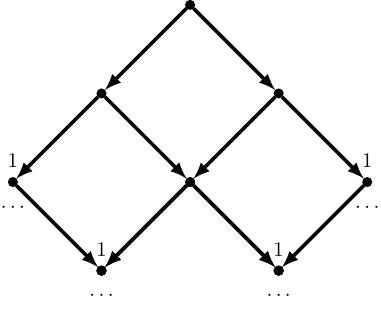} }}
    \caption{Example of unlabeled chip-firing on an infinite directed lattice of nonnegative integer points on the plane.}%
    \label{fig:exampleunlabel}
\end{figure}
\end{example}

We claim that if we start a chip-firing game with the initial configuration as above, it will always terminate.
\begin{lemma}\label{lem:terminate}
    Let $n \geq 0$. Suppose we start with $2^n$ chips at the root. Then, the chip-firing game must terminate.
\end{lemma}
\begin{proof}
Suppose, on the contrary, that we never reach a stable configuration. Let $t_i$ be the total number of chips received by vertices in row $i$, with $t_0=2^n$. Since each chip visits each row at most once, only finitely many firings can occur in any fixed finite set of rows. Hence an infinite firing sequence must reach every row, so $t_i\geq1$ for every nonnegative integer $i$. Since vertices in row $i+1$ receive chips from only vertices in row $i$, we know that $t_{i} \geq t_{i+1}$. Since $t_i \geq t_{i+1}$ and $t_i \geq 1$ for all natural numbers $i$, we know that there exist integers $t \geq 1$ and $N \geq 0$ such that $t_i=t$ for all $i \geq N$. Since $t_i = t_{i+1}=t$ for $i \geq N$, no chip stays in row $i$.

For $i \geq N$, let $a_i$ denote the number of chips that are received by the leftmost vertex on row $i$ that receives chips. This vertex fires all its chips. Its left child is the leftmost vertex receiving chips in row $i+1$ and receives chips only from it. Thus, $a_{i+1}=a_i/2$, giving an infinite strictly decreasing sequence of positive integers, a contradiction.
\end{proof}

We observe that, as long as we start with the same number of chips at the root, all firing processes lead to the same stable configuration. The following lemma proves this directly for our infinite graph.

\begin{lemma}\label{lem:lattice-confluence}
    Let $n \geq 0$. Suppose we have two firing processes $\mathcal{F}$ and $\mathcal{G}$ that
    \begin{enumerate}
        \item start with the initial configuration with $2^n$ chips at the root and no chips elsewhere and
        \item fire as long as there is a legal firing move that can be done.
        \end{enumerate}
        Then, $\mathcal{F}$ and $\mathcal{G}$ lead to the same stable configuration, and every vertex fires the same number of times in both processes.
\end{lemma}
\begin{proof}
    From Lemma~\ref{lem:terminate}, we know that $\mathcal{F}$ and $\mathcal{G}$ both terminate and yield a stable configuration. A vertex receiving a total of $M$ chips, including any initial chips, fires $\floor{M/2}$ times and retains $M\bmod 2$ chips. For the root, $M=2^n$. For a vertex in any later row, $M$ is determined by the firing counts of its parents in the preceding row. Thus, induction on rows shows that every vertex receives the same number of chips, fires the same number of times, and retains the same number of chips in both processes.
\end{proof}

\section{The Total Number of Firings}\label{sec:TotalFirings}

Let us denote by $T(n)$ the total number of firings when we start with $2^n$ chips at the origin. By Lemma~\ref{lem:lattice-confluence}, this number does not depend on the order of firings.

We calculated the first few values of this sequence starting from index 0 (the new sequence A389565 in the OEIS \cite{oeis}):
\[0,\ 1,\ 5,\ 15,\ 52,\ 163,\ 458,\ 1359,\ 4296,\ 12890,\ 38570.\]

We would like to introduce the distribution of the number of chips in the stable configuration as a function of $y-x$. In other words, we denote by $D_n(i)$ the number of chips in all the vertices $(x,y)$ in the stable configuration such that $y-x = i$ when we start with $2^n$ chips. We call this distribution \textit{the distance distribution}.

\begin{example}
\label{ex:n=4D}
    We get that the distance distribution $D_4$ is
    \[2,\ 1,\ 2,\ 3,\ 0,\ 3,\ 2,\ 1,\ 2.\] To see this, one can either do direct computation of the stable configuration or consider the table in Example~\ref{ex:n=4}, which is the table for the total number of chips reaching each vertex in the firing process for $n=4$. The odd values in bold correspond to vertices that have one chip in the stable configuration, since each firing removes two chips from the vertex that is fired. One can count the bold entries in each column to obtain the distance distribution $D_4$.
\end{example}

Note that the distance distribution is symmetric, and for $n \ge 1$ its middle value is zero. Consider the unnormalized second raw moment $\mu_2'(n)$ of the distribution $D_n(i)$:
\[\mu_2'(n) = \sum_i i^2D_n(i).\]

\begin{proposition}
    We have the following connection:
    \[\mu_2'(n) = 2T(n).\]
\end{proposition}

\begin{proof}
    The second raw moment of the initial distribution is 0. We can consider the second raw moments of the resulting chip distribution after some number of firings. Suppose we have 2 chips at the vertex $(x,y)$. Before firing, they contribute $2(y-x)^2$ to the second moment. After the firing, we have one chip at $(x+1,y)$ and the other at $(x,y+1)$; they contribute
    \[(y-x-1)^2 + (y-x+1)^2 = 2(y-x)^2 + 2.\]
    Thus, each firing increases the second moment by 2. The result follows.
\end{proof}

\begin{example}
    From Example~\ref{ex:n=4D}, we get that
    \[T(4) = \frac{\mu_2'(4)}{2} = \frac{2\cdot 4^2 + 3^2 + 2\cdot 2^2 + 3\cdot 1 + 3\cdot 1 + 2\cdot 2^2+ 3^2 + 2\cdot 4^2}{2}  = 52.\]
\end{example}

Figure~\ref{fig:Distancedistributionfor15} shows the distance distribution for $n=15$.

\begin{figure}[ht!]
    \centering
\begin{tikzpicture}[scale=0.5, x=0.18cm,y=0.015cm]
  \draw[->] (-0.5,0) -- (100,0) node[below] {$i$};
  \draw[->] (0,-30) -- (0,580) node[left] {$D_{15}(i)$};

\foreach \k/\distance in {0/-45,15/-30,30/-15,45/0,60/15,75/30,90/45} {
  \draw (\k,0) -- (\k,-12)  node[below] {\scriptsize $\distance$};
}

\foreach \v in {100,200,...,500} {
  \draw (0,\v) -- (-1.5,\v) node[left] {\scriptsize $\v$};
}

  \foreach \k/\v in {
    0/545,1/517,2/489,3/461,4/433,5/406,6/380,7/355,8/332,9/310,
    10/290,11/271,12/254,13/253,14/272,15/282,16/292,17/300,18/289,19/309,
    20/312,21/299,22/353,23/358,24/389,25/411,26/425,27/439,28/474,29/495,
    30/514,31/437,32/433,33/427,34/458,35/426,36/423,37/407,38/380,39/360,
    40/321,41/277,42/232,43/186,44/108,45/0,
    46/108,47/186,48/232,49/277,50/321,51/360,52/380,53/407,54/423,55/426,
    56/458,57/427,58/433,59/437,60/514,61/495,62/474,63/439,64/425,65/411,
    66/389,67/358,68/353,69/299,70/312,71/309,72/289,73/300,74/292,75/282,
    76/272,77/253,78/254,79/271,80/290,81/310,82/332,83/355,84/380,85/406,
    86/433,87/461,88/489,89/517,90/545
  }{
    \fill (\k,\v) circle (1.4pt);
  }

  \draw
    (0,545)--(1,517)--(2,489)--(3,461)--(4,433)--(5,406)--(6,380)--(7,355)--(8,332)--(9,310)--
    (10,290)--(11,271)--(12,254)--(13,253)--(14,272)--(15,282)--(16,292)--(17,300)--(18,289)--(19,309)--
    (20,312)--(21,299)--(22,353)--(23,358)--(24,389)--(25,411)--(26,425)--(27,439)--(28,474)--(29,495)--
    (30,514)--(31,437)--(32,433)--(33,427)--(34,458)--(35,426)--(36,423)--(37,407)--(38,380)--(39,360)--
    (40,321)--(41,277)--(42,232)--(43,186)--(44,108)--(45,0)--(46,108)--(47,186)--(48,232)--(49,277)--
    (50,321)--(51,360)--(52,380)--(53,407)--(54,423)--(55,426)--(56,458)--(57,427)--(58,433)--(59,437)--
    (60,514)--(61,495)--(62,474)--(63,439)--(64,425)--(65,411)--(66,389)--(67,358)--(68,353)--(69,299)--
    (70,312)--(71,309)--(72,289)--(73,300)--(74,292)--(75,282)--(76,272)--(77,253)--(78,254)--(79,271)--
    (80,290)--(81,310)--(82,332)--(83,355)--(84,380)--(85,406)--(86,433)--(87,461)--(88,489)--(89,517)--(90,545);
\end{tikzpicture}
    \caption{Distance distribution for $n=15$.}
\label{fig:Distancedistributionfor15}
\end{figure}

\section{The Intermediate Firing Configuration}\label{sec:IntermedFiringConfig}

\subsection{General Facts}

For this section, we stabilize the chip-firing process on the lattice of nonnegative integer points using the following strategy: starting from $i=0$ and for each $i$, in increasing order, we repeatedly fire vertices in row $i$ until no vertex in that row can fire anymore. In our study of chip-firing, it suffices to consider just this procedure, since Lemma~\ref{lem:lattice-confluence} tells us that, regardless of the sequence of firings used to reach a stable configuration, the number of times each vertex fires is the same, and the final stable configuration is the same.

For the chip-firing game starting with $2^n$ chips at the root, we define $F_n(x,y)$ to be the number of chips at the vertex $(x,y)$ after all rows with an index lower than $x+y$ have finished firing and before the vertices in row $x+y$ start to fire. Note that if the vertex $(x,y)$ never receives chips, then $F_n(x,y) = 0$. For $(x, y) \in \mathbb{Z}^2$ not in $\mathbb{Z}_{\geq 0}^2$, we assume that $F_n(x, y) = 0$. We call the table $F_n$ of values of $F_n(x, y)$ the \textit{intermediate firing configuration}. For $i\in\mathbb{Z}$, the \emph{row $i$ of $F_n$} is
$\bigl(F_n(i-j,\,j)\bigr)_{j\in\mathbb{Z}}$. We say that $F_n(a,b)$, the entry of $F_n$ for $(a, b)$, is to the \textit{left} of $F_n(c, d)$, the entry of $F_n$ for $(c, d)$, if $(a, b)$ is to the left of $(c, d)$.

\begin{example}
\label{ex:n=4}
  Consider the intermediate firing configuration $F_4$. The nonzero entries of the intermediate firing configuration are as illustrated in Figure~\ref{fig:F4}, with odd values in bold. Row $2$ of $F_4$ is $4, 8, 4$, with $F_4(2, 0) = F_4(0, 2)=4$ and $F_4(1, 1)=8$.
  \begin{figure}[ht!]
\begin{center}
  \begin{tabular}{ccccccccc}
        & & &  & 16  & & &  \\
      &  & & 8 &  &  8 & &  \\  
      & & 4 &  & 8 & &  4 & \\ 
      & 2 & & 6 & & 6 & & 2 \\ 
      \textbf{1} & & 4 & & 6 & & 4  &  & \textbf{1} \\ 
      & 2 & & \textbf{5} & & \textbf{5}  & & 2 \\ 
      \textbf{1} & & \textbf{3} & & 4 & & \textbf{3} & & \textbf{1} 
      \\  & \textbf{1} & & \textbf{3} & & \textbf{3} & & \textbf{1} \\ & & \textbf{1} & & 2 & & \textbf{1} \\ 
      & & & \textbf{1} & & \textbf{1} & & &
  \end{tabular}   
\end{center}
\caption{Nonzero entries of the intermediate firing configuration $F_4$.}
    \label{fig:F4}
\end{figure}
\end{example}

The intermediate firing configuration is connected to the stable configuration as follows.

\begin{proposition}
\label{prop:stableisodd}
If $F_n(x,y)$ is odd, then, in the stable configuration, vertex $(x, y)$ has one chip; otherwise, it has zero chips.
\end{proposition}

\begin{proof}
    Note that in our strategy, vertex $(x, y)$ receives $F_n(x, y)$ chips and then fires until it is no longer possible. Since the outdegree of each vertex is $2$, vertex $(x, y)$ fires $\floor{\frac{F_n(x, y)}{2}}$ times and loses $2\floor{\frac{F_n(x, y)}{2}}$ chips. Thus, in the stable configuration, it has $F_n(x, y) - 2\floor{\frac{F_n(x, y)}{2}}$ chips. The result immediately follows.
\end{proof}

\begin{proposition}
    The last row in the stable configuration is the same as the last nonzero row in $F_n$.
\end{proposition}

\begin{proof}
Let $i$ be the index of the last row of $F_n$ with nonzero entries. The nonzero entries of $F_n$ on that row must be equal to $1$; otherwise, there is a vertex in the $i$th row that can fire, moving chips to the next row. It follows that the $i$th row of the stable configuration is the same as the last nonzero row of $F_n$.
\end{proof}

\subsection{Symmetry}

The intermediate firing configuration is symmetric with respect to swapping $x$ and $y$. In other words, $y=x$ is the line of symmetry. We have the following proposition.

\begin{proposition}\label{prop:symmetry}
The intermediate firing configuration obtained is symmetric over the line $y=x$. In other words,
\[F_n(x,y) = F_n(y,x).\]
\end{proposition}

In particular, we have the following result.

\begin{proposition}
\label{prop:evenCentre}
    Let $x \geq 0$ and $n \geq 1$. Then, $F_n(x, x)$ is even.
\end{proposition}
\begin{proof}
    First note $F_n(0, 0) = 2^n$, which is even because $n \geq 1$. For the rest of the proof, assume that $x \geq 1$.

    We observe that $F_n(x, x)  = \floor{\frac{F_n(x, x-1)}{2}} + \floor{\frac{F_n(x-1, x)}{2}} = 2\floor{\frac{F_n(x, x-1)}{2}}$ with the last equality resulting from Proposition~\ref{prop:symmetry}, which implies $F_n(x, x-1) = F_n(x-1, x)$. This completes the proof.
\end{proof}

The final stable configuration is symmetric with respect to the line $y=x$ too. We have the following corollary from Proposition~\ref{prop:stableisodd} and Proposition~\ref{prop:evenCentre}.

\begin{corollary}
    For $n \ge 1$, no chip in the stable configuration ends up on the line $y=x$. 
\end{corollary}

\subsection{Computed Example and Observations}

Figure~\ref{fig:stableconfign=9} shows the stable configuration for $n=9$, together with the vertices the chips pass through. Only vertices $(x, y)$ with $F_n(x,y) > 0$ are drawn. By Proposition~\ref{prop:stableisodd}, a vertex keeps one chip in the stable configuration exactly when $F_n(x,y)$ is odd; these are the filled dots. Unfilled dots are vertices where $F_n(x,y)$ is even and nonzero: chips pass through them, but none remain. To fit the page, the diagram is rotated so that chips move from left to right. Row $i$ of $F_9$ appears as a vertical column indexed by $i$, with the origin in column~$0$. The line $y=x$ is the horizontal axis of symmetry.

\begin{figure}[ht!]
    \centering
\includegraphics[width=0.95\linewidth]{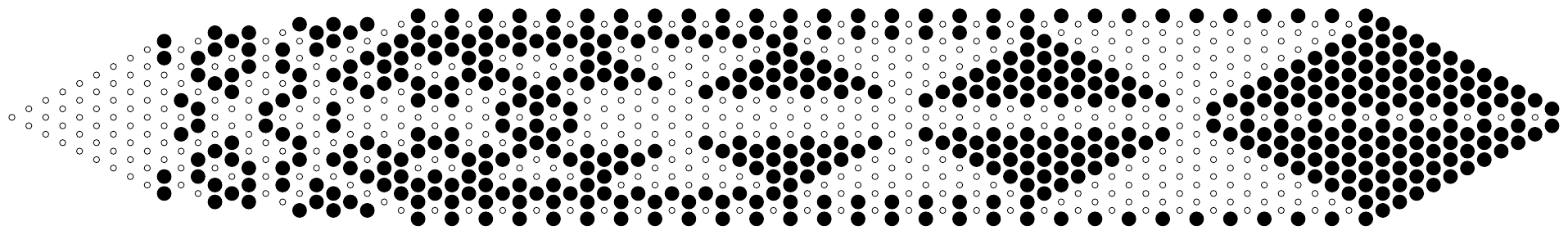}
    \caption{The stable configuration for $n=9$. Filled, bold dots represent vertices that have a single chip in the stable configuration. Unfilled dots represent vertices that receive chips but retain none.}
\label{fig:stableconfign=9}
\end{figure}

Figure~\ref{fig:stableconfign=9} illustrates several patterns that we see in intermediate firing configurations; they guide our discussions later. We observe the following from Figure~\ref{fig:stableconfign=9}:
\begin{itemize}
    \item We have a left triangle, corresponding to the first nine rows. It consists only of unfilled dots.
    \item Then the structure's width increases a bit in a transition region. After that, we have a long rectangular region whose columns alternate between $12$ and $13$ dots. Filled and unfilled dots alternate along its upper and lower boundaries. The rectangular region contains the longest rows of $F_9$.
    \item The triangle on the right consists of filled dots except along the horizontal axis of symmetry.
    \item The pattern of filled and unfilled dots seems to be more chaotic in the transition region. In the rectangular region, the pattern becomes more regular on the right. The patterns in the left and right triangles can be explicitly described.
\end{itemize}

The left triangle is explained in Section~\ref{sec:PascalTriangle}, the middle region is studied in Section~\ref{sec:Med}, and the right triangle is described in Section~\ref{sec:BottomTriangle}.

\section{Pascal's Triangle and the Intermediate Firing Configuration}\label{sec:PascalTriangle}

When we start with $2^n$ chips at the root, the top $n+1$ rows of the intermediate firing configuration are especially easy to describe. As we can see in the next theorem, for each $i \in \{0, 1, \dots, n\}$, the $i$th row of Pascal's triangle, scaled by $2^{n-i}$, appears as row $i$ of the intermediate firing configuration $F_n$.

\begin{theorem}\label{thm:PascalTriangle}
If we start with $2^n$ chips at the root, then for each $i \in \{0, 1, \dots, n\}$ and $x \in \{0, 1, \dots, i\}$, we have
\[F_n(x,i-x) = 2^{n-i} \binom{i}{x}.\]
\end{theorem}

\begin{proof}
    We use induction on $i$, with the convention that $\binom{i}{x}=0$ for $x\notin\{0,1,\ldots,i\}$. Consider the base case $i=0$. In the initial configuration of the chip-firing game, there are $2^n = 2^{n-i}$ chips at vertex $(0, 0)$ of row $0$, which matches our statement. 

    Suppose the formula holds for some row $i < n$. All entries in row
$i$ are even. Thus, for $0\leq x\leq i+1$, we have
    \[
\begin{aligned}
F_n(x,i+1-x)
&=\frac{F_n(x,i-x)+F_n(x-1,i+1-x)}{2}\\
&=2^{n-i-1}\left(\binom{i}{x}+\binom{i}{x-1}\right) = 2^{n-i-1}\binom{i+1}{x}.
\end{aligned}
\]
This concludes the induction.
    \end{proof}

In terms of parameters $(x,y)$, vertex $(x,y)$ with $x+y \le n$ has $2^{n-x-y} \binom{x+y}{x}$ chips.

\begin{corollary}
Suppose we start with $2^n$ chips at the root. In the final stable configuration, the closest row to the root that contains chips is the $n$th row. 
\end{corollary}
    
\begin{proof}
     From Theorem~\ref{thm:PascalTriangle}, we know that in the intermediate firing configuration, the number of chips at each vertex on row $m < n$ is divisible by $2^{n-m}$, implying that row $m$ does not have any chips in the stable configuration. However, row $n$ equals the $n$th row of the Pascal triangle; in particular, it starts with 1, implying by Proposition~\ref{prop:stableisodd} that there are chips in row $n$ of the stable configuration.
\end{proof}

This corollary explains the left triangle from Figure~\ref{fig:stableconfign=9} consisting only of unfilled dots.

\section{The Middle of the Intermediate Firing Configuration}\label{sec:Med}

For the middle of the intermediate firing configuration table $F_n$, it is hard to write the entries in a clean formula. After row $n$, the rows are no longer scaled versions of those of Pascal's triangle. However, we find and establish here several relations between the entries of each row of $F_n$.

\subsection{General Structure of Rows}

In the next theorem, we show that, in the left part of the intermediate firing configuration, the row entries increase by at least two, except at the very end, when the second entry is exactly in the middle; then the increase might be only by one.

\begin{theorem}\label{thm:IncreaseUntilMiddle}
    Let $ n\geq 0$. Each row of $F_n$, starting with its leftmost nonzero entry, increases until the middle. Moreover, if $F_n(x+1,y-1) > 0$, and $y < x$, then
    \begin{equation}F_n(x, y) - F_n(x+1, y-1) \ge 2.\end{equation}
    If $F_n(x+1,y-1)>0$ and $y = x$, then
    \begin{equation}F_n(x, y) - F_n(x+1, y-1) \ge 1.\end{equation}
    By symmetry, the entries decrease in the right half.
\end{theorem}
\begin{proof}
For $n=0$, the assertion is immediate. Assume $n\geq1$. We prove the theorem via induction. For row $i=0$, the statement is true.

We now prove the inductive step. Suppose the statement is true for row $i-1$. Consider two points $(x+1,y-1)$ and $(x,y)$ on row $i$, where $y\leq x$ and $F_n(x+1,y-1)>0$. Given that $F_n(x+1,y-1) > 0$, it follows that $F_n(x, y-1) > 0$. If $F_n(x+1, y-2) = 0$, then $F_n(x, y-1) > 1$ for $F_n(x+1,y-1)$ to get any chips. In any case, $F_n(x, y-1) - F_n(x+1, y-2) \ge 2$.

Suppose $y < x-1$. Then $F_n(x-1, y)$ is to the left of the middle; then, by the induction hypothesis, $ F_n(x-1, y) - F_n(x, y-1) \ge 2$. Combining this with $F_n(x, y-1) - F_n(x+1, y-2) \ge 2$, we deduce that $F_n(x-1, y) - F_n(x+1, y-2) \ge 4$, implying that
\[F_n(x, y) - F_n(x+1, y-1) = \floor{\frac{F_n(x-1, y)}{2}} - \floor{\frac{F_n(x+1, y-2)}{2}} \ge 2,\]
finishing the induction step.

Now, suppose $F_n(x-1, y)$ is in the middle, or equivalently, $y = x-1$. Then the inductive hypothesis implies $F_n(x-1, y)- F_n(x, y-1) \ge 1$. Combining this with $F_n(x, y-1) - F_n(x+1, y-2) \ge 2$, we obtain $F_n(x-1, y) - F_n(x+1, y-2) \ge 3$. Adding the fact from Proposition~\ref{prop:evenCentre} that $F_n(x-1, y)$ is even, we obtain, again, 
\[F_n(x, y) - F_n(x+1, y-1) = \floor{\frac{F_n(x-1, y)}{2}} - \floor{\frac{F_n(x+1, y-2)}{2}} \ge 2.\]

Now, we prove the last case that if $y = x$, then $F_n(x, y) - F_n(x+1, y-1) \ge 1$.

If $x=y$, then $F_n(x, y) = \floor{\frac{F_n(x-1, y)}{2}} + \floor{\frac{F_n(x, y-1)}{2}} = 2\floor{\frac{F_n(x, y-1)}{2}}$. Note that $F_n(x+1, y-1) = \floor{\frac{F_n(x+1, y-2)}{2}} + \floor{\frac{F_n(x, y-1)}{2}}$. As shown at the beginning of the inductive step, we have $F_n(x,y-1)-F_n(x+1,y-2)\geq2$. Consequently, we have
\[F_n(x,y) - F_n(x+1, y-1) = \floor{\frac{F_n(x, y-1)}{2}} - \floor{\frac{F_n(x+1, y-2)}{2}} \ge 1.\]
By symmetry of entries in $F_n$, we get the final claim, concluding the proof.
\end{proof}

\subsection{Row Length}

Now we discuss the lengths of each row of the intermediate firing configuration, defined as its number of nonzero entries. We describe an example with $n = 9$ to match Figure~\ref{fig:stableconfign=9}. 

\begin{example}
For $F_9$, in the top triangle, the row length starts with 1 and increases by 1 until it reaches $n+1 = 10$. This is our top triangle. Then, we have $13$ rows where alternating terms slowly and weakly increase:
\[9,\ 10,\ 11,\ 10,\ 11,\ 10,\ 11,\ 12,\ 11,\ 12,\ 11,\ 12,\ 11.\]
This is the top of the middle part. Then, we have our rectangular shape with 57 rows that alternate between 12 and 13, starting with 12. In the bottom triangle, the row length decreases by 1, from 13 to 2.
The total number of rows is $10+13+57+12 = 92$.\end{example}

In the example above, we see that the numbers of elements in two consecutive rows differ by 1.

\begin{theorem}
If rows $i$ and $i+1$ of $F_n$ are both nonzero, then the number of nonzero elements in row $i$ of $F_n$ differs from the number of nonzero elements in row $i+1$ of $F_n$ by exactly $+1$ or $-1$.
\end{theorem}
\begin{proof}
Let $F_n(x, y)$ be the leftmost nonzero entry of the $i$th row. Consequently, the rightmost nonzero entry of the $i$th row is $F_n(y, x) = F_n(x,y)$. Then, by Theorem~\ref{thm:IncreaseUntilMiddle}, all entries between $(x, y)$ and $(y, x)$ are nonzero, implying that there are $x-y+1$ nonzero terms in row $i$.

We consider two cases. First, suppose $F_n(x, y) \ge 2$.  We find that $F_n(x+1, y) = \floor{\frac{F_n(x, y)}{2}} + \floor{\frac{F_n(x+1, y-1)}{2}} \geq 1$, and $F_n(x+1, y)$ is the leftmost nonzero term of row $i+1$, implying that there are $x+2-y$ nonzero terms in row $i+1$. 

Now consider the case where $F_n(x, y) = 1$. We find that $F_n(x+1, y) = \floor{\frac{F_n(x, y)}{2}} =0$ and $F_n(x,y+1) = \floor{\frac{F_n(x, y)}{2}} + \floor{\frac{F_n(x-1, y+1)}{2}} \geq 1$. The last inequality holds because, due to Theorem~\ref{thm:IncreaseUntilMiddle}, we have $F_n(x-1,y+1) > F_n(x,y) = 1$. Since row $i+1$ is nonzero and the boundary entries of row $i$ are both $1$, row $i$ has at least three nonzero entries. Thus, $x-y\geq2$, so $(x-1,y+1)$ is neither the leftmost nor the rightmost nonzero entry of row $i$. We thus deduce that $F_n(x, y+1)$ is the leftmost nonzero entry of row $i+1$. Therefore, there are exactly $x-y-1+1= x-y$ nonzero entries in row $i+1$.
\end{proof}

It follows that whenever the length of the row decreases, the previous row has one chip at each end in the stable configuration. 

In the middle of the intermediate firing configuration, we see a lot of cases when two rows that are two apart are the same length.

\begin{proposition}
\label{prop:rowstarts1a}
Let $n\geq 0$. If row $i$ of $F_n$ starts with $1$ and $a$, where $4 \le a \le 7$, then row $i+2$ has the same length and also starts with $1$.
\end{proposition}

\begin{proof}
    Suppose the leftmost nonzero element in row $i$ is $F_n(x,y) = 1$, implying that the row length is $x-y+1$. Assume also that $F_n(x-1,y+1) = a$, where $4 \le a \le 7$. Then the vertex $(x,y)$ does not fire, and the leftmost nonzero term in row $i+1$ is $F_n(x,y+1) = \floor{\frac{F_n(x-1,y+1)}{2}}$. We have $2 \le \floor{\frac{F_n(x-1,y+1)}{2}} = \floor{\frac{a}{2}} \le 3$. It follows that the leftmost element in row $i+2$ is at the vertex $(x+1,y+1)$ and is equal to $\floor{\frac{F_n(x,y+1)}{2}} = 1$. It also follows that the length of row $i+2$ is $(x+1) - (y+1) + 1 = x-y+1$, the same as the length of row $i$.
\end{proof}

We also computed the length of the longest row, starting from index 0 (new sequence A390355 in the OEIS \cite{oeis}):
\[1,\ 2,\ 3,\ 4,\ 5,\ 6,\ 7,\ 8,\ 10,\ 13,\ 15,\ 19,\ 24,\ 30,\ 37,\ 46,\ 58,\ 73,\ \ldots .\]

\begin{example}
    We see that the sequence above starts as an arithmetic progression. Proposition~\ref{prop:rowstarts1a} explains why for $n \le 7$, the last row of the top triangle, which is the same as a row in the Pascal triangle, might be the longest row in the intermediate firing configuration.
\end{example}

Our computer data show that the length of the longest row is the width of the rectangle, which is the bottom part of the middle section.

\subsection{The Number of Nonzero Rows}

We start by looking at central elements.

\begin{proposition}\label{prop:Decreasing}
    For $n\geq1$, if $F_n(x, x) > 0$, then $F_n(x+1, x+1) \leq F_n(x, x)-2$.
\end{proposition}
\begin{proof}
Suppose $F_n(x,x) > 0$. We also know that $F_n(x,x)$ is even by Proposition~\ref{prop:evenCentre}, implying that $F_n(x+1, x) = F_n(x, x+1) = \floor{\frac{F_n(x+1, x-1)}{2}} + \frac{F_n(x, x)}{2}$.
     
Now, by Theorem~\ref{thm:IncreaseUntilMiddle}, we have $F_n(x+1, x-1) < F_n(x, x)$. Because of this and since $F_n(x, x)$ is even, we obtain that $\frac{F_n(x,x)}{2} >\floor{\frac{F_n(x+1, x-1)}{2}}$. Consequently, we find that
\[F_n(x+1, x+1) = \floor{\frac{F_n(x+1, x)}{2}} + \floor{\frac{F_n(x, x+1)}{2}} = 2\floor{\frac{\floor{\frac{F_n(x+1, x-1)}{2}} + \frac{F_n(x, x)}{2}}{2}} < F_n(x, x).\]

Knowing that both $F_n(x, x)$ and $F_n(x+1, x+1)$ are even, we find $F_n(x+1, x+1) \leq F_n(x, x)-2$.
\end{proof}

We establish an upper bound on the number of rows.

\begin{proposition}\label{prop:FiniteBoundOnRows}
    The number of nonzero rows of $F_n$ is at most $n+\binom{n}{n/2}$ if $n$ is even, and at most $n+1+2\floor{\binom{n}{\lfloor n/2\rfloor}/2}$ if $n$ is odd.
\end{proposition}
\begin{proof}
    For $n=0$, the only occupied row is row $0$, so the bound holds. We now consider even $n\geq2$. By Theorem~\ref{thm:PascalTriangle}, the central element in row $i=n$ is $F_n(\frac{n}{2}, \frac{n}{2}) = \binom{n}{n/2}$. Using the fact that $F_n(x+1, x+1) \leq F_n(x, x) - 2$ from Proposition~\ref{prop:Decreasing}, and also the fact that $\binom{n}{n/2}$ is even for $n > 0$,
    we get
    \[F_n\left(\frac{n}{2}+\frac{\binom{n}{n/2}}{2}, \frac{n}{2}+\frac{\binom{n}{n/2}}{2}\right) \le 0.\]
    As the central element in this row is zero, the entire row $n+\binom{n}{n/2}$ consists of zeros.

    Now consider the case where $n$ is odd. We find that near the center of row $i=n$, we have $F_n(\lceil n/2 \rceil, \floor{n/2}) = F_n( \floor{n/2}, \lceil n/2 \rceil) = \binom{n}{\floor{n/2}}$. Similar to the above, we get that
    \[F_n\left(\frac{n+1}{2}+\floor{\binom{n}{\floor{n/2}}/2}, \frac{n+1}{2}+\floor{\binom{n}{\floor{n/2}}/2}\right) \le 0,\]
    implying that row $n+1 +2\floor{\binom{n}{\lfloor n/2 \rfloor}/2}$ consists entirely of zeros.
\end{proof}

One can observe that if row $i$ of $F_n$ is nonzero, then all previous rows are nonzero. To see this, we observe that all vertices in row $i$ received all of their chips from row $i-1$, since all in-neighbors of row $i$ are from row $i-1$. It follows that if $F$ has $x$ nonzero rows, then the last row is row $x-1$.

We wrote a program to calculate the number of nonzero rows in the table $F_n$, starting from $n=0$ (new sequence A390129 in the OEIS \cite{oeis}):

\[1,\ 2,\ 4,\ 6,\ 10,\ 16,\ 24,\ 38,\ 60,\ 92,\ 144,\ 226,\ 362,\ 570,\ 906,\ 1430,\ \ldots.\]

Notice that starting from index $n=1$, the number of rows is even. The following sequence is half of the above for $n > 0$:
\[1,\ 2,\ 3,\ 5,\ 8,\ 12,\ 19,\ 30,\ 46,\ 72,\ 113,\ 181,\ 285,\ 453,\ 715,\ \ldots.\] 

This is not a coincidence, as we prove in the following proposition.
\begin{proposition}
    The number of nonzero rows in the intermediate firing configuration $F_n$ is even for $n\geq 1$.
\end{proposition}
\begin{proof}
Every entry in the last nonzero row is at most $1$, since otherwise the next row would be nonzero. If its index were even, its central entry would be positive by Theorem~\ref{thm:IncreaseUntilMiddle}, and hence at least $2$ by Proposition~\ref{prop:evenCentre}, a contradiction. Thus, the last nonzero row has odd index. Since the nonzero rows are consecutive and start at row $0$,
their number is even.
\end{proof}

\section{The Bottom Triangle}\label{sec:BottomTriangle}

Throughout this section, assume $n\geq1$. This section is devoted to the bottom rows of the table $F_n$, which we also call the bottom triangle, where the lengths of the rows strictly decrease as one reads down the table. 

For $j\geq1$, we call a row of nonzero entries of length $j+1$ \textit{minimal} if it has the smallest possible entries, is symmetric, and satisfies the conditions of Theorem~\ref {thm:IncreaseUntilMiddle}. We denote such a row as $R(j)$. 

We can see that if $j = 2k-1$, then the row $R(j)$ is 
\[1,\ 3,\ 5,\ 7,\ \dots,\ j,\ j,\ \dots,\ 7,\ 5,\ 3,\ 1.\]
When $j = 2k$, the row $R(j)$ is
\[1,\ 3,\ 5,\ 7,\ \dots,\ j-1,\ j,\ j-1,\ \dots,\ 7,\ 5,\ 3,\ 1.\]

We define \textit{the bottom triangle} as the largest set of rows at the bottom of $F_n$ such that each next row has one fewer entry than the previous one. We will later show that the bottom triangle consists of rows $R(j)$, where $R(1)$ is the last row.

Let us calculate the number of chips in row $R(j)$. We denote this sum as $S(j)$. We see from examples that the sequence $S(j)$ starts as 
\[2,\ 4,\ 8,\ 12,\ 18,\ 24,\ 32,\ 40,\ 50,\]
and is sequence A007590 in the OEIS \cite{oeis} with an index shift ($S(j) = A007590(j+1)$). We have $S(2k-1) = 2k^2$ and $S(2k) = 2k(k+1)$. Combining the formulas, we get
\[S(j) = \floor{\frac{(j+1)^2}{2}}.\]

\begin{lemma}\label{lem:NextRow}
If a row in an intermediate firing configuration is $R(j)$ with $j\geq2$, then the next row is $R(j-1)$. The row $R(1)=(1,1)$ is followed by a zero row.
\end{lemma}
\begin{proof}
The number of chips that stay at the vertices in row $R(j)$ after firing is $2\floor{\frac{j+1}{2}}$. Thus, the number of chips that move to the next row is $\floor{\frac{(j+1)^2}{2}} - 2\floor{\frac{j+1}{2}}$.

If $j=2k-1$, the above equals $2k^2 - 2k = 2k(k-1) = S(2k-2) = S(j-1)$. If $j=2k$, the above equals $2k^2 +2k - 2k = 2k^2 = S(2k-1) = S(j-1)$. We know that the length of a row can decrease by at most 1, and the only row with at least $j$ entries with $S(j-1)$ chips is the minimal row $R(j-1)$.
\end{proof}

\begin{proposition}
    If in an intermediate firing configuration $F_n$, a row is $R(j)$ and the previous row has $j+2$ nonzero elements, then the previous row is $R(j+1)$.
\end{proposition}

\begin{proof}
The previous row has $j+2$ nonzero entries, so it contains at least $S(j+1)$ chips. After firing, it retains at most $2\floor{(j+2)/2}$ chips, since its central entry, when present, is even by Proposition~\ref{prop:evenCentre}.

It sends exactly $S(j)$ chips to the next row. Therefore, it contains at most
\[
S(j)+2\floor{\frac{j+2}{2}}=S(j+1)
\]
chips. Thus, it contains exactly $S(j+1)$ chips and is the minimal row $R(j+1)$.
\end{proof}

The last nonzero row consists entirely of ones. By Theorem~\ref{thm:IncreaseUntilMiddle}, it has at most two nonzero entries, and it cannot have one nonzero entry because a positive central entry is even. Hence it is $R(1)=(1,1)$. Applying the preceding proposition repeatedly shows that the bottom triangle consists solely of minimal rows.

For the larger values of $n$ in our computations up to and including $12$, we observe that any intermediate firing configuration $F_n$ ends with the long rectangle followed by the bottom triangle. We also see that the rectangle is the widest part of the intermediate firing configuration. This inspires the following conjecture.

\begin{conjecture}
    The number of rows in the bottom triangle of $F_n$ equals the length of the longest row of $F_n$ minus 1.
\end{conjecture}

\section{Discussion}\label{sec:Discussion}

Figure~\ref{fig:threerows} superimposes a plot of nonzero entries of row $11$, a plot of nonzero entries of row $48$, and a plot of nonzero entries of row $208$ for $F_{11}$. For each of those rows, we plot $F_{11}(x,y)$ against $y-x$. Below, for each row, we list the nonzero entries of the row from left to right in order of decreasing $x$, so the first entry
is plotted farthest to the left.

The dashed curve is the plot of nonzero entries of row~$11$, the last row of the top triangle:
\begin{align*}
   \bigl(F_{11}(11,0),\ F_{11}(10,1),\ F_{11}(9,2),\ & \dots,\ F_{11}(2,9),\ F_{11}(1,10),\ F_{11}(0,11)\bigr)\\
 =(1,\ 11,\ 55,\ 165,\ 330,\ 462,\ & 462,\ 330,\ 165,\ 55,\ 11,\ 1).
\end{align*}

Overall, we know from Theorem~\ref{thm:PascalTriangle} that the nonzero entries in row $n$ of $F_n$ are $(F_n(n, 0), F_n(n-1, 1), \dots, F_n(0, n)) = (\binom{n}{0}, \binom{n}{1}, \dots, \binom{n}{n})$, the $n$th row of Pascal's triangle. When we plot the entries of row $n$ of $F_n$ against $y-x$ and appropriately rescale both axes, the shape of the plot approaches the plot of the normal density as $n\to\infty$:
\[{\varphi (s)={\frac {e^{-s^{2}/2}}{\sqrt {2\pi }}}}.\]

The solid black curve is the plot of nonzero entries of row~$48$, the first row of maximal length:
\begin{align*}
    \bigl(F_{11}(33,15),\ F_{11}(32,16),\ F_{11}(31,17),\ & \dots,\ F_{11}(17,31),\ F_{11}(16,32),\ F_{11}(15,33)\bigr)
\\ =(1,\ 6,\ 18,\ 38,\ 66,\ 102,\ 143,\ 181,\ 208,\ & 218,\ 208,\ 181,\ 143,\ 102,\ 66,\ 38,\ 18,\ 6,\ 1).
\end{align*}

The gray curve is the plot of nonzero entries of row~$208$, the first row of the bottom triangle:
\begin{align*}
 \bigl(F_{11}(113,95),\ F_{11}(112,96),\ F_{11}(111, 97),\ & \dots,\ F_{11}(97, 111),\ F_{11}(96, 112),\ F_{11}(95,113)\bigr)\\
=(1,\ 3,\ 5,\ 7,\ 9,\ 11,\ 13,\ 15,\ 17,\ & 18,\ 17,\ 15,\ 13,\ 11,\ 9,\ 7,\ 5,\ 3,\ 1).
\end{align*} The gray curve, excluding the center point, consists of two line segments and looks almost flat.

\begin{figure}[ht!]
    \centering
\begin{tikzpicture}[x=0.55cm,y=0.01cm]
\draw[<->] (-0.5,0) -- (20,0) node[below] {$y-x$};
\draw[->] (9,-20) -- (9,520) node[above] {$F_{11}(x,y)$};

\foreach \tickpos/\ticklabel in
  {0/-18,4.5/-9,13.5/9,18/18} {
  \draw (\tickpos,0) -- ++(0,-3pt)
    node[below] {\scriptsize $\ticklabel$};
}

\node[below left] at (9,0) {\scriptsize $0$};

\foreach \tickvalue in {100,200,300,400,500} {
  \draw (9,\tickvalue) -- ++(-3pt,0)
    node[left] {\scriptsize $\tickvalue$};
}
  \foreach \k in {0,1,...,18} {
  }

  \foreach \k/\v in {
    0/1,1/3,2/5,3/7,4/9,5/11,6/13,7/15,8/17,9/18,10/17,11/15,12/13,13/11,14/9,15/7,16/5,17/3,18/1
  }{
    \fill[gray] (\k,\v) circle (2pt);
  }

    \draw[gray]
    (0,1)--(1,3)--(2,5)--(3,7)--(4,9)--(5,11)--(6,13)--(7,15)--(8,17)--(9,18)--
    (10,17)--(11,15)--(12,13)--(13,11)--(14,9)--(15,7)--(16,5)--(17,3)--(18,1);

  \foreach \k/\v in {
    0/1,1/6,2/18,3/38,4/66,5/102,6/143,7/181,8/208,9/218,
    10/208,11/181,12/143,13/102,14/66,15/38,16/18,17/6,18/1
  }{
    \fill (\k,\v) circle (2pt);
  }

  \draw
    (0,1)--(1,6)--(2,18)--(3,38)--(4,66)--(5,102)--(6,143)--(7,181)--(8,208)--(9,218)--
    (10,208)--(11,181)--(12,143)--(13,102)--(14,66)--(15,38)--(16,18)--(17,6)--(18,1);


  \foreach \k/\v in {3.5/1,4.5/11,5.5/55,6.5/165,7.5/330,8.5/462,9.5/462,10.5/330,11.5/165,12.5/55,13.5/11,14.5/1} {
    \fill (\k,\v) circle (2pt);
  }
  
  \foreach \k/\v in {
    3.5/1,4.5/11,5.5/55,6.5/165,7.5/330,8.5/462,
    9.5/462,10.5/330,11.5/165,12.5/55,13.5/11,14.5/1
  }{
  }

  \draw[black,dashed]
    (3.5,1)--(4.5,11)--(5.5,55)--(6.5,165)--(7.5,330)--(8.5,462)--
    (9.5,462)--(10.5,330)--(11.5,165)--(12.5,55)--(13.5,11)--(14.5,1);
\end{tikzpicture}   
    \caption{Plots of rows 11 (dashed black), 48 (solid black), and 208 (solid gray) in $F_{11}$. }
    \label{fig:threerows}
\end{figure}

If we look at the representation of the stable configuration in Figure~\ref{fig:stableconfign=9}, we see that the top of the middle part is more chaotic, while it gets more and more structured in the lower rows. The bottom triangle is very structured, so that we can even completely describe it through minimal rows. Not surprisingly, this description correlates with Figure~\ref{fig:threerows}, where the top plot, which resembles the bell curve, is associated with randomness, and the bottom plot, whose two sides are arithmetic progressions, is associated with structure.

We want to understand this change of structure better, so we look at the difference tables in the next section.

\section{Difference Tables}\label{sec:DiffTables}

We now define the table of differences between consecutive entries in $F_n$. For $n \geq 0$, we define the \emph{difference table} $F'_n$ as follows:
\[F'_n(x,y) = F_n(x-1,y) - F_n(x,y-1).\]
\begin{example}
    For example,
    \[F'_n(1,0) = F_n(0,0) - F_n(1,-1) = 2^n\]
    and
    \[F'_n(0,1) = F_n(-1,0) - F_n(0,0) = -2^n.\]
\end{example}
\textit{Row $i$ of the difference table $F_n'$} is the sequence $(F_n'(i-j, j))_{j \in \mathbb{Z}}$,  the sequence of values of $F_n'$ on $x+y=i$ in order of decreasing $x$.

\begin{example}\label{ex:Beginning_of_table}
Consider the top of the table $F_n$ for $n>3$ as depicted in Figure~\ref{fig:Table1}. Then, we can compute the top of the corresponding difference table $F'_n$, as depicted in Figure~\ref{fig:Table2}.
\begin{figure}[ht!]
\begin{center}
  \begin{tabular}{ccccccccc}
        & & &  & $2^n$  & & &  \\
      &  & & $2^{n-1}$ &  &  $2^{n-1}$ & &  \\  
      & & $2^{n-2}$ &  & $2^{n-1}$ & &  $2^{n-2}$ & \\ 
      & $2^{n-3}$ & & $3 \cdot 2^{n-3}$ & & $3 \cdot 2^{n-3}$ & & $2^{n-3}$ \\ 
      \textbf{$2^{n-4}$} & & $4 \cdot 2^{n-4}$ & & $6 \cdot 2^{n-4}$ & & $4 \cdot 2^{n-4}$  &  & \textbf{$2^{n-4}$} \\ 
  \end{tabular} 
\end{center}
  \caption{Top of the table $F_n$ for $n > 3$ (only nonzero values).}
\label{fig:Table1}
\end{figure}
\begin{figure}[ht!]
\begin{center}
  \begin{tabular}{ccccccccccc}
      & & & & $2^n$ & & $-2^n$  & & & & \\
      & & & $2^{n-1}$ & & 0 & & $-2^{n-1}$ & & & \\  
      & & $2^{n-2}$ & & $2^{n-2}$ & & $-2^{n-2}$ & & $-2^{n-2}$ & &\\ 
      & $2^{n-3}$ & & $2 \cdot 2^{n-3}$ & & 0 & & $-2 \cdot 2^{n-3}$ & & $-2^{n-3}$ & \\ 
      $2^{n-4}$ & & $3 \cdot 2^{n-4}$ & & $2 \cdot 2^{n-4}$ & & $-2 \cdot 2^{n-4}$ & & $-3 \cdot 2^{n-4}$  &  & $-2^{n-4}$ \\ 
  \end{tabular}
\end{center}
\caption{Top of the table $F'_n$ for $n > 3$ (with zeros outside the main shape omitted).}
      \label{fig:Table2}
\end{figure}
\end{example}

We see that for $n > 3$, the maximum absolute values in rows of $F'_n$, as we can see from Figure~\ref{fig:Table2}, are
\[2^n,\ 2^{n-1},\ 2^{n-2},\ 2^{n-2},\ 3\cdot 2^{n-4}.\]
We see that this sequence of maxima is nonincreasing, and we will prove that later. We will also prove that in the left half of each row of the difference table, the values weakly increase, and then weakly decrease.

When we compute the difference tables $F'_n$ for some large $n$, we notice consecutive entries that have the same value. Towards the bottom rows, we observe that these regions look like triangles. We illustrate this phenomenon for the difference table $F'_9$ in Figure~\ref{fig:diff9}. Row $i$ appears as the $i$th column; the leftmost column is row~$1$ of the difference table. Each dot is an entry of $F'_9$. We omit the zero entries outside the main shape. A filled dot marks an entry that equals an adjacent entry in the same row, so runs of filled dots are runs of equal values. 
\begin{figure}[ht!]
    \centering
    \includegraphics[width=1\linewidth]{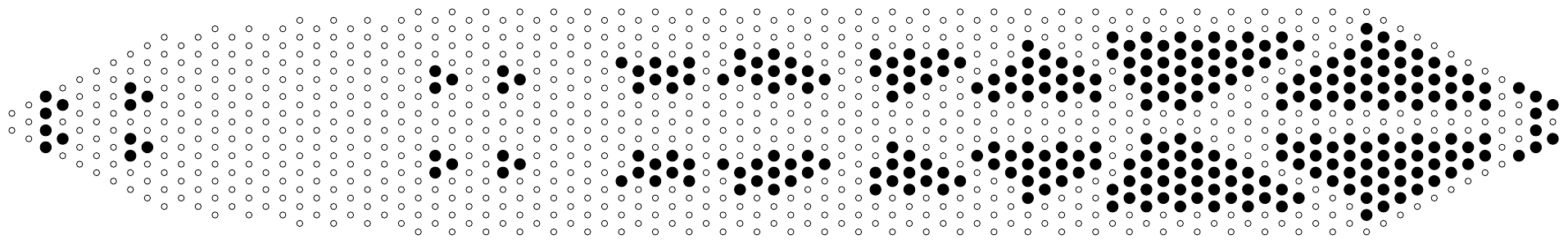}
\caption{The difference table $F'_9$, with one dot per entry (zero entries outside the main shape omitted). Filled dots: entries equal to an adjacent entry in the same row; runs of filled dots are runs of equal values. Row $i$ is the $i$th column; the leftmost column represents row~$1$ of the difference table.}
\label{fig:diff9}
\end{figure}

We now prove general properties of difference tables. The first is analogous to Proposition~\ref{prop:symmetry}. The symmetry of the intermediate firing configuration implies the antisymmetry of the corresponding difference table.
\begin{proposition}[Antisymmetry]\label{prop:DiffTableSym}
Let $n\geq 0$. For all integers $x$ and $y$, we have $F'_n(x,y) = -F'_n(y,x)$.
\end{proposition}

It is not surprising that the maximum element in rows of the intermediate firing configuration is not increasing. The next proposition shows that the same is true for difference tables, explaining that the rows of the intermediate firing configuration are becoming flatter. Due to antisymmetry and Theorem~\ref{thm:IncreaseUntilMiddle}, the maximum absolute value in a row of the difference table is the same as the maximum value on the left side of the row.

\begin{proposition}\label{prop:nonincreasingmax}
For $n \ge 0$ and $i \ge 1$, the maximum element of the $(i+1)$st row of $F'_n$ is no greater than that of row $i$ of $F'_n$.
\end{proposition}

\begin{proof}
Let $a=F_n(x_0, y_0)$, $b=F_n(x_0-1, y_0+1)$, and $c=F_n(x_0-2, y_0+2)$ be three consecutive entries of row $i-1$ of $F_n$. By definition of $F_n'$, we have $F_n'(x_0, y_0+1) = b-a$ and $F_n'(x_0-1, y_0+2)= c-b$. By the firing rule, we have that the entry of $F_n'$ for $(x_0, y_0+2)$, which is on row $i+1$, is
\[F_n'(x_0, y_0+2)
= F_n(x_0-1, y_0+2) - F_n(x_0, y_0+1)=
\floor{\frac{c}{2}}-\floor{\frac{a}{2}}
\leq \ceil{\frac{c-a}{2}}
\leq \max\{b-a,c-b\}.
\]
The last inequality holds because $b-a$ and $c-b$ are integers whose sum is $c-a$. Since every entry of row $i+1$ arises this way, the result follows.
\end{proof}

Each row of the intermediate firing configuration $F_n$ behaves similarly to a Gaussian probability density function $\varphi(x) = \frac{1}{\sqrt{2\pi}}e^{-\frac{x^2}{2}}$. Like the Gaussian function for $x < 0$, the entries of $F_n$ increase as we move rightwards towards the center. In the left half of each row of the difference table $F'_n$, the entries first weakly increase and then weakly decrease. The behavior of the row of $F'_n$ is similar to the behavior of the derivative of the Gaussian $\varphi'(x) = -\frac{x}{\sqrt{2\pi}}e^{-\frac{x^2}{2}}$.

\begin{example}
Figure~\ref{fig:threerows_diffs} superimposes plots of rows $12$,  $49$, and $209$ of the difference table $F'_{11}$, with zeros outside the main shape omitted. For each of these rows, we plot $F'_{11}(x, y)$ against $y-x$.  Below, for each row, we list the entries from the first nonzero entry to the last, from left to right in order of decreasing $x$, so the first entry is plotted farthest to the left.

The entries of row $12$ of $F'_{11}$ (dashed curve), from the first nonzero entry to the last, are:
\begin{align*}
   (F'_{11}(12, 0),\ F'_{11}(11, 1),\ F'_{11}(10, 2),\ & \dots,\ F'_{11}(2, 10),\ F'_{11}(1, 11),\ F'_{11}(0, 12))\\
   =  (1,\ 10,\ 44,\ 110,\ 165,\ 132,\ & 0,\ -132,\ -165,\ -110,\ -44,\ -10,\ -1).
\end{align*}
The nonzero entries of row 49 of $F'_{11}$ (black curve) are:
\begin{align*}
    (F'_{11}(34, 15),\ F'_{11}(33, 16),\ F'_{11}(32,17),\ & \dots,\ F'_{11}(17, 32),\ F'_{11}(16, 33),\ F'_{11}(15, 34))  \\
    = (1,\ 5,\ 12,\ 20,\ 28,\ 36,\ 41,\ 38,\ 27,\ 10,\ & -10,\ -27,\ -38,\ -41,\ -36,\ -28,\ -20,\ -12,\ -5,\ -1).
\end{align*}
The nonzero entries of row 209 of $F'_{11}$ (gray curve) are:
\begin{align*}
(F'_{11}(114, 95),\ F'_{11}(113, 96),\ F'_{11}(112, 97),\ & \dots,\ F'_{11}(97, 112),\ F'_{11}(96, 113),\ F'_{11}(95, 114)) \\
=  (1,\ 2,\ 2,\ 2,\ 2,\ 2,\ 2,\ 2,\ 2,\ 1,\ & -1,\ -2,\ -2,\ -2,\ -2,\ -2,\ -2,\ -2,\ -2,\ -1).\end{align*}
    \begin{figure}[ht!]
\centering
\begin{tikzpicture}[scale=0.8,x=0.55cm,y=0.02cm]
\draw[<->] (-1.2,0) -- (20,0) node[below] {$y-x$};
\draw[<->] (9,-190) -- (9,190) node[above] {$F'_{11}(x,y)$};

\foreach \tickpos/\ticklabel in
  {0/-18,4.5/-9,13.5/9,18/18} {
  \draw (\tickpos,0) -- ++(0,-3pt)
    node[below] {\scriptsize $\ticklabel$};
}

\node[below left] at (9,0) {\scriptsize $0$};

\foreach \tickvalue in {-150,-100,-50,50,100,150} {
  \draw (9,\tickvalue) -- ++(-3pt,0)
    node[left] {\scriptsize $\tickvalue$};
}

  \foreach \x/\v in {
    -0.5/1,0.5/2,1.5/2,2.5/2,3.5/2,4.5/2,5.5/2,6.5/2,7.5/2,8.5/1,9.5/-1,10.5/-2,11.5/-2,12.5/-2,13.5/-2,14.5/-2,15.5/-2,16.5/-2,17.5/-2,18.5/-1
  }{
    \fill[gray] (\x,\v) circle (2pt);
  }
  \draw[gray]
    (-0.5,1)--(0.5,2)--(1.5,2)--(2.5,2)--(3.5,2)--(4.5,2)--(5.5,2)--(6.5,2)--(7.5,2)--(8.5,1)--
    (9.5,-1)--(10.5,-2)--(11.5,-2)--(12.5,-2)--(13.5,-2)--(14.5,-2)--(15.5,-2)--(16.5,-2)--(17.5,-2)--(18.5,-1);

  \foreach \x/\v in {
    -0.5/1,0.5/5,1.5/12,2.5/20,3.5/28,4.5/36,5.5/41,6.5/38,7.5/27,8.5/10,9.5/-10,10.5/-27,11.5/-38,12.5/-41,13.5/-36,14.5/-28,15.5/-20,16.5/-12,17.5/-5,18.5/-1
  }{
    \fill (\x,\v) circle (2pt);
  }
  \draw
    (-0.5,1)--(0.5,5)--(1.5,12)--(2.5,20)--(3.5,28)--(4.5,36)--(5.5,41)--(6.5,38)--(7.5,27)--(8.5,10)--
    (9.5,-10)--(10.5,-27)--(11.5,-38)--(12.5,-41)--(13.5,-36)--(14.5,-28)--(15.5,-20)--(16.5,-12)--(17.5,-5)--(18.5,-1);

  \foreach \x/\v in {
  3.0/1,4.0/10,5.0/44,6.0/110,7.0/165,8.0/132,9.0/0,10.0/-132,11.0/-165,12.0/-110,13.0/-44,14.0/-10,15.0/-1
  }{
    {
    \fill (\x,\v) circle (2pt);
  }
  }
  \draw[black,dashed]
    (3.0,1)--(4.0,10)--(5.0,44)--(6.0,110)--(7.0,165)--(8.0,132)--(9.0,0)--(10.0,-132)--(11.0,-165)--(12.0,-110)--(13.0,-44)--(14.0,-10)--(15.0,-1);
\end{tikzpicture}
\caption{Plots of rows $12$ (dashed black), $49$ (solid black),
and $209$ (solid gray) in $F'_{11}$.}
\label{fig:threerows_diffs}
\end{figure}
\end{example}

\begin{remark}
    One can observe that the entries do not necessarily first strictly increase and afterward strictly decrease. To see this, consider the table of differences of consecutive entries in each row of the difference table $F'_9$, which we illustrate in Figure~\ref{fig:plusminustable}. The $i$th column in the figure corresponds to the differences between consecutive entries in the $i$th row of $F'_9$. Here each dot corresponds to a pair of consecutive entries in a row of $F'_9$ and records the sign of their difference: a bold black dot if the entry increases, a hollow dot if the two entries are equal, and a gray dot if it decreases. This is analogous to viewing the sign of the second derivative of the probability density function $\varphi(s) = \frac{e^{-\frac{s^2}{2}}}{\sqrt{2\pi}}$.
\end{remark}

\begin{figure}[ht!]
    \centering    \includegraphics[width=1.0\linewidth]{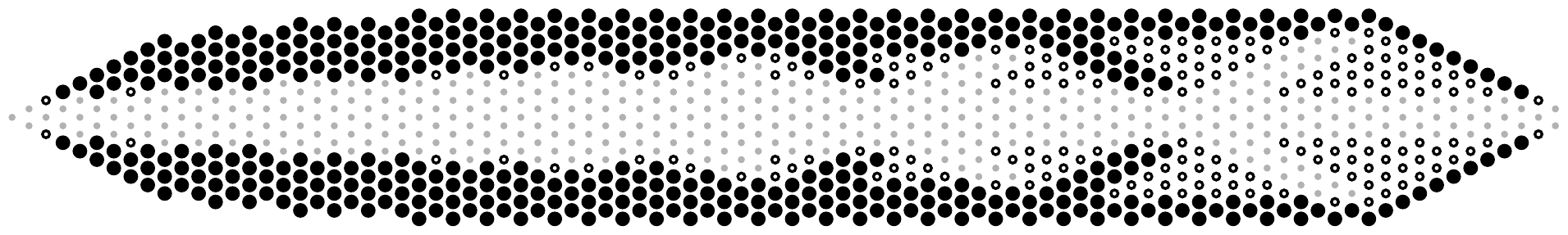}
    \caption{Illustration of the sign of differences between consecutive entries in $F'_9$. Bold black dots signify increases in entries in $F'_9$, hollow dots signify two consecutive entries that are equal, and gray dots signify decreases in entries.}
    \label{fig:plusminustable}
\end{figure}

The following lemma shows that if three terms in the difference table weakly increase/decrease, then the two terms below also weakly increase/decrease.

\begin{lemma}\label{lem:IncreasingDiffTableSecondDeriv}
Let $x_0,y_0\in\mathbb Z$ with $x_0+y_0\geq0$. If $F'_n(x_0, y_0+1) \leq F'_n(x_0-1, y_0+2) \leq F'_n(x_0-2, y_0+3)$, then $F'_n(x_0, y_0+2) \leq F'_n(x_0-1, y_0+3)$. Similarly, if $F'_n(x_0, y_0+1) \ge F'_n(x_0-1, y_0+2) \ge F'_n(x_0-2, y_0+3)$, then $F'_n(x_0, y_0+2) \ge F'_n(x_0-1, y_0+3)$.
\end{lemma}

\begin{proof}
Write the four consecutive entries $F_n(x_0-r,y_0+r)$, for $r=0,1,2,3$, as $a,b,c,d$, respectively, and put $m=c-b$. The first hypothesis says that $b-a\leq m\leq d-c$, so $c\leq a+2m$ and $d\geq b+2m$. Since $m$ is an integer,
\[
\floor{\frac{c}{2}}-\floor{\frac{a}{2}}
\leq m
\leq \floor{\frac{d}{2}}-\floor{\frac{b}{2}}.
\]
By the firing rule, the outer expressions are $F'_n(x_0,y_0+2)$ and $F'_n(x_0-1,y_0+3)$, respectively. This proves the first assertion. Reversing the inequalities proves the second.
\end{proof}

\begin{theorem}
    Let $n \geq 0$. In the left half of each row of the difference table $F'_n$, the entries weakly increase and then weakly decrease.
\end{theorem}

\begin{proof}
We prove this via induction on the row number. 

For the base case,  row $1$ is $(2^n,-2^n)$, so its left half consists of a single entry and the statement holds.

We can consider our rows to be infinite; we can assume that they start with zeros and end with zeros. Adding zeros at the beginning does not change the property that the row starts as weakly increasing because nonzero terms at the beginning of the row are positive. For our inductive case, we suppose that the left half of row $i$ of the difference table first increases weakly, then decreases weakly. We observe that from Lemma~\ref{lem:IncreasingDiffTableSecondDeriv}, if we have three consecutive weakly increasing entries in the difference table, then, directly below, we have two consecutive weakly increasing entries in the next row.  Similarly, if we have three consecutive weakly decreasing entries in the difference table, then, directly below, we have two consecutive weakly decreasing entries in the next row. Choose a peak in the left half of row $i$. The lemma shows that the entries in the next row weakly increase up to the first entry directly below this peak and weakly decrease from the second entry directly below it onward, considering only entries in the left half. The comparison between these two entries, if both occur, is the only comparison not determined by the lemma. Either ordering gives a sequence that weakly increases and then weakly decreases. If the next row has a central zero, appending it preserves this property, since the entries in the left half are nonnegative.
\end{proof}

\section*{Acknowledgments} 

The authors thank Professor Alexander Postnikov for suggesting the topic of chip-firing on posets, for helping formulate the project, and for helpful discussions. The MIT Department of Mathematics financially supports the first and second authors. The third author is supported by Harvard College.

All figures in this paper were generated using TikZ via \url{mathcha.io} and exported as PDF.

\section*{AI Usage}

After the paper was written, we used ChatGPT Pro and Claude to check for mistakes.

\bibliographystyle{unsrt}
\bibliography{biblio5.bib}

\smallskip

\noindent
Ryota Inagaki \\
\textsc{
Department of Mathematics, Massachusetts Institute of Technology\\
77 Massachusetts Avenue, Building 2, Cambridge, MA 02139, U.S.A.}\\
\textit{E-mail address: }\texttt{inaga270@mit.edu}
\medskip

\noindent
Tanya Khovanova \\
\textsc{
Department of Mathematics, Massachusetts Institute of Technology\\
77 Massachusetts Avenue, Building 2, Cambridge, MA 02139, U.S.A.}\\
\textit{E-mail address: }\texttt{tanyakh@yahoo.com}
\medskip

\noindent
Austin Luo \\
\textsc{
Harvard College,\\
1 Oxford Street, Cambridge, MA 02138, U.S.A.}\\
\textit{E-mail address: }\texttt{aluo@college.harvard.edu}
\medskip

\end{document}